\documentclass[11pt]{article}
\usepackage[T1]{fontenc}
\usepackage[utf8]{inputenc}
\usepackage[margin=1in]{geometry}
\usepackage{amsmath,amsthm,amssymb,mathtools}
\usepackage{newtxtext,newtxmath}
\usepackage{microtype}
\usepackage{booktabs,array}
\usepackage{aliascnt}
\usepackage[authoryear,round]{natbib}

\usepackage{enumitem}
\usepackage[hidelinks]{hyperref}
\usepackage[capitalize,noabbrev]{cleveref}
\hypersetup{pdftitle={The pros and cons of Cover's randomization for testing by betting},
  pdfsubject={Bell--Cover competitive optimality, numeraire e-variables, and sequential testing},
  pdfkeywords={e-value, e-process, numeraire, log-optimality, randomization, testing by betting}}
\setlist[enumerate]{itemsep=0.15em,topsep=0.3em}
\allowdisplaybreaks[1]
\numberwithin{equation}{section}
\newtheorem{theorem}{Theorem}[section]
\newaliascnt{proposition}{theorem}
\newtheorem{proposition}[proposition]{Proposition}
\aliascntresetthe{proposition}
\crefname{proposition}{Proposition}{Propositions}
\Crefname{proposition}{Proposition}{Propositions}
\newaliascnt{lemma}{theorem}
\newtheorem{lemma}[lemma]{Lemma}
\aliascntresetthe{lemma}
\crefname{lemma}{Lemma}{Lemmas}
\Crefname{lemma}{Lemma}{Lemmas}
\newaliascnt{corollary}{theorem}
\newtheorem{corollary}[corollary]{Corollary}
\aliascntresetthe{corollary}
\crefname{corollary}{Corollary}{Corollarys}
\Crefname{corollary}{Corollary}{Corollarys}
\theoremstyle{definition}
\newaliascnt{definition}{theorem}
\newtheorem{definition}[definition]{Definition}
\aliascntresetthe{definition}
\crefname{definition}{Definition}{Definitions}
\Crefname{definition}{Definition}{Definitions}
\newaliascnt{example}{theorem}
\newtheorem{example}[example]{Example}
\aliascntresetthe{example}
\crefname{example}{Example}{Examples}
\Crefname{example}{Example}{Examples}
\theoremstyle{remark}
\newaliascnt{remark}{theorem}
\newtheorem{remark}[remark]{Remark}
\aliascntresetthe{remark}
\crefname{remark}{Remark}{Remarks}
\Crefname{remark}{Remark}{Remarks}
\newcommand{\E}{\mathbb E}
\newcommand{\Pcal}{\mathcal P}
\newcommand{\Ecal}{\mathcal E}
\newcommand{\Fcal}{\mathcal F}
\newcommand{\Ccal}{\mathcal C}
\newcommand{\ind}{\mathbf 1}
\newcommand{\Unif}{\operatorname{Unif}}

\newcommand{\pos}[1]{\left(#1\right)_{+}}
\newcommand{\num}{E^{\star}}
\newcommand{\comp}{\mathrm{C}}
\newcommand{\calcomp}{\mathrm{C,cal}}
\newcommand{\umi}{\mathrm{UM}}
\newcommand{\ord}{\mathrm{ord}}
\newcommand{\doi}[1]{doi:\,\href{https://doi.org/#1}{\nolinkurl{#1}}}

\title{Competitive optimality in testing by betting\\ via Bell-Cover randomization}
\author{Aaditya Ramdas \\ \texttt{aramdas@stanford.edu}}
\date{\today}

\begin{document}
\maketitle
\vspace{-2.8em}
\begin{abstract}
Bell and Cover showed that an investor who multiplies the initial unit of capital by an independent uniform random variable on $(0,2)$, and then uses the log-optimal portfolio, wins a head-to-head wealth comparison with probability at least one half against every independently randomized competitor. We explain very simply how this result transfers to testing by betting: for any composite null $\mathcal P$ and simple alternative $Q$, denoting $E^*$ as the corresponding numeraire e-variable, we show that $UE^*$ exceeds any other e-variable $E$ with probability at least half.
% ; the essential property is the numeraire inequality $\E_Q[E/\num]\leq1$. 
Interestingly, we show that this competitive optimality result is actually equivalent to the numeraire inequality $\E_Q[E/\num]\leq1$, and in general randomization only helps the numeraire and fails to improve the competitive advantage of an arbitrary e-variable. 
Under optional stopping with or without knowledge of $U$, we emphasize a key distinction between e-process validity and competitive optimality. 
% For finite positive numeraires, the randomization incurs a fixed expected log loss; it neither confers competitive optimality on an arbitrary e-variable nor provides a uniform power advantage. 
We also show that competitive optimality comes at the price of expected log wealth and power: thresholding $UE^*$ at $1/\alpha$ has sharp size at most $\alpha/2$, but the factor of two actually disappears under optional stopping. Even after correcting for this factor of two, the test is dominated in conditional rejection probability by randomizing the testing threshold (randomized Markov's inequality). Thus, Bell-Cover randomization is optimal for a specific competitive objective, at the cost of others.
\end{abstract}

\section{Introduction}\label{sec:history}
In game-theoretic statistics~\citep{shafer2021testing,ramdas2023game,RamdasWang2025}, a betting-based test typically starts with one unit of capital. A statistician uses this capital to bet against a given null hypothesis; the resulting wealth after one bet is called an e-value, and repeated this across multiple rounds leads to wealth processes (test supermartingales, or more generally for composite nulls, e-processes). The larger the wealth, the more evidence there is against the null.

An initially surprising alternative is to exchange the initial unit wealth for a random initial wealth using a fair \emph{lottery}
\begin{equation}\label{eq:lottery}
 U\sim\Unif(0,2),\qquad \E[U]=1,
\end{equation}
and then follow the same betting strategy starting with initial wealth $U$. Whether this helps depends on the objective. There are three different questions: does the resulting capital remain valid evidence against the null; does it beat another evidence-producing strategy; and does it improve the probability of rejecting the null? These questions have different answers that we will provide in this paper.

The following competitive result is due to \citet[Theorem~1]{BellCover1980}:  an investor who starts with capital $U$ and then uses the log-optimal portfolio, wins a head-to-head wealth comparison with probability at least one half against every admissible independently randomized competitor. We will translate this result to the modern testing by betting context, and show quite simply that that the same conclusion holds for the log-optimal numeraire e-value, which was recently shown by~\cite{LarssonEtAl2025} to always exist. Fascinatingly, we will show that competitive optimality is actually equivalent to the numeraire property.
% It is separate from the universal-portfolio theorem of \citet{Cover1991}, whose comparison is with the best constant-rebalanced portfolio in hindsight and whose central guarantee concerns asymptotic log-wealth along market sequences.

% The relevant history goes substantially beyond one-period portfolios. 
\citet[Theorem~2]{BellCover1988} proved, for a \emph{general convex family} of nonnegative random wealths, the equivalence between relative log-optimality and the expected-ratio inequality that is now called the numeraire property. Their Theorem~3 separates a competitive portfolio game into a fair-lottery problem and a log-optimal allocation problem, their Theorem~4 treats multiple periods, and their Section~6 considers conditional and terminal randomization. 

 \citet{LarssonEtAl2025} were aware of this early work on log-optimality and the numeraire; when introducing Proposition~2.4, they  point to \citet[Theorem~16.2.2]{CoverThomas2006}. The novelty of the more recent work is to establish existence of a numeraire for the \emph{entire class of e-variables} of an arbitrary  composite null, against an arbitrary point alternative, with no domination or finite-information assumptions. 
 Their theory permits infinite values and develops the associated effective-null and reverse-information-projection duality. 
% An equivalence of optimality properties for an existing candidate is not an existence theorem for that statistical class. 

% These results also belong to a broader literature. The process-level numeraire is a standard object in mathematical finance; see, for example, \citet{KaratzasKardaras2007}; the difference in new work is the focus on arbitrary composite nulls without dominating reference measures. Growth-optimal e-testing and reverse information projections were developed in a more limited capacity before the unrestricted numeraire theorem in \citet{GrunwaldEtAl2024} and \citet{LardyEtAl2024}. 

% One final historical point anticipates the conclusion. 
\citet[Section~5, p.~165]{BellCover1980} explicitly distinguish protection in their competitive game from increasing investment capital and decline to advocate the additional lottery in practice. Our paper also emphasizes the distinction between this particular form of game-theoretic competitive optimality and practical statistical benefit, and is therefore consistent with the older interpretations.

\section{Statistical setting and the numeraire}\label{sec:setting}

Let $(\Omega,\Fcal)$ be the underlying measure space, let $\Pcal$ be a nonempty family of probability measures representing the null, and let $Q$ be a specified alternative probability measure. Define
\begin{equation}\label{eq:eclass}
 \Ecal(\Pcal)
 =\{E:\Omega\to[0,\infty]\text{ measurable}:\ \E_P[E]\leq1\text{ for all }P\in\Pcal\},
\end{equation}
which is the set of all e-variables for $\Pcal$; an observed realization is called an e-value. The class is convex and contains the constant one.
% All logarithms in this paper are natural.

\begin{definition}[Numeraire]\label{def:num}
An e-variable $\num$ is called a  numeraire for $\Pcal$ against $Q$ if it is strictly positive $Q$-almost surely and
\begin{equation}\label{eq:num}
 \E_Q\!\left[\frac{E}{\num}\right]\leq1
 \quad\text{for every }E\in\Ecal(\Pcal).
\end{equation}
\end{definition}
% We distinguish ordinary finite ratios from extended-value conventions. 
Unless explicitly stated otherwise, in this paper, competitive comparisons further assume
\begin{equation}\label{eq:finite}
 \num<\infty\qquad Q\text{-almost surely}.
\end{equation}
Then \eqref{eq:num} also forces every competing e-variable to be finite $Q$-almost surely. 
Following~\cite{LarssonEtAl2025}, this is equivalent to assuming that $Q\ll\Pcal$, which means
\[
 \bigl[P(A)=0\text{ for every }P\in\Pcal\bigr]\ \Longrightarrow\ Q(A)=0,
 \qquad A\in\Fcal.
\]
This is much weaker than requiring $Q\ll P_0$ for some particular $P_0\in\Pcal$. All it says is that if some event is impossible under every $P$, then it should be impossible under $Q$; if this condition is not met for some $A$, then clearly one can set the e-variable equal to infinity on that $A$, and pay no price under the null, so this condition only serves to exclude ``trivialities'' when it is obvious that one can reject the null by immediately making an infinite amount of money by betting only a dollar. We summarize one of their results below.
% The following result summarizes two theorems from~\cite{LarssonEtAl2025}.

\begin{theorem}\label{thm:lrr}
For every nonempty $\Pcal$ and every $Q$, a numeraire exists and is unique up to $Q$-null sets. Moreover, the numeraire is finite $Q$-almost surely if and only if $Q\ll\Pcal$. 
% Equivalently, every e-variable is finite $Q$-almost surely, or the e-variable class is bounded in $Q$-probability.
\end{theorem}

This is Theorem~2.6 of \citet{LarssonEtAl2025}, together with their uniqueness observation.  
Their proof is non-elementary; the competitive and testing arguments below are proved directly.

\paragraph{Private randomization and the comparison probability.}

A privately randomized e-variable is a measurable $T(\omega,z)\geq0$, where the seed $Z$ has a fixed law $\rho$, independent of the data under every null and alternative under consideration, and
\begin{equation}\label{eq:random-e}
 \E_{P,\rho}[T]\leq1\qquad(P\in\Pcal).
\end{equation}
Only this ex ante condition is required: fixing a particular seed need not leave an e-variable. 
By Tonelli's theorem,
\begin{equation}\label{eq:average}
 \overline T(\omega):=\int T(\omega,z)\,\rho(dz)\in\Ecal(\Pcal).
\end{equation}
Two competing players use independent private seeds, but naturally evaluate their statistics on the same data. Whenever auxiliary randomization is present, $\Pr_Q$ and $\E_Q$ include it unless a conditioning statement says otherwise. In particular, the half-probability guarantee below is under $Q$ and the lotteries jointly, not conditional on every realized dataset.

\section{The Bell--Cover mechanism on a convex class}\label{sec:mechanism}

The following elementary calculation contains the randomization argument.

\begin{lemma}[Uniform comparison identity]\label{lem:uniform}
Let $0<S^{\star}<\infty$ and $S\geq0$ be random variables, and let $U\sim\Unif(0,2)$ be independent of both. Denoting their ratio as $R=S/S^{\star}$, we have
\begin{equation}\label{eq:identity}
 \Pr_Q(S\geq US^{\star})
 =\E_Q\!\left[\min\left\{\frac{R}{2},1\right\}\right].
\end{equation}
Consequently, if $\E_Q[R]\leq1$, then
\begin{equation}\label{eq:half}
 \Pr_Q(US^{\star}>S)\geq\frac12.
\end{equation}
More generally, for every $c>0$,
\begin{equation}\label{eq:relative-tail}
 \Pr_Q(S\geq cUS^{\star})
 \leq \min\left\{1,\frac{1}{2c}\right\}.
\end{equation}
\end{lemma}
\begin{proof}
Condition on $(S,S^{\star})$. The probability that $U\leq R$ is $\min\{R/2,1\}$. Taking expectations proves \eqref{eq:identity}, and $\min\{R/2,1\}\leq R/2$ gives \eqref{eq:half}. Replace $R$ by $R/c$ to obtain \eqref{eq:relative-tail}. Independence and continuity of $U$ remove finite ties.
\end{proof}

The bound \eqref{eq:relative-tail} is sharp: take $S=S^{\star}$. The non-randomized bound is $\Pr_Q(S\geq cS^{\star})\leq\min\{1,1/c\}$. In the portfolio setting these are precisely the two types of comparison bounds given in \citet[Corollaries~1--2]{BellCover1980}.

\subsection{Log-optimality, expected ratios, and competitive optimality}

\begin{theorem}[Three equivalent properties]\label{thm:equivalence}
Let $\Ccal$ be a convex family of nonnegative, $Q$-almost surely finite random variables, and let $S^{\star}\in\Ccal$ be strictly positive $Q$-almost surely. The following are equivalent:
\begin{enumerate}[label=(\roman*)]
\item $\E_Q[S/S^{\star}]\leq1$ for every $S\in\Ccal$.
\item $\E_Q[\log(S/S^{\star})]\leq0$ for every $S\in\Ccal$, where each expectation is well defined in $[-\infty,0]$.
\item For every $S\in\Ccal$ and an independent $U\sim\Unif(0,2)$,  we have $\Pr_Q(US^{\star}>S)\geq1/2$.
\end{enumerate}
\end{theorem}

The equivalence (i)--(ii) is \citet[Theorem~2]{BellCover1988}, also covered in~\cite{LarssonEtAl2025}. We include the brief argument below for completeness, and note that it is only their equivalence to (iii) that is new.

\begin{proof}
Assume (i). For $R=S/S^{\star}$, $\log^+R\leq R$ is integrable, and $\log R\leq R-1$. Thus $\E_Q[\log R]$ is well defined, possibly $-\infty$, and is at most zero. This proves (ii). Also, \cref{lem:uniform} immediately gives (iii).

Next suppose (ii), fix $S\in\Ccal$, and write $Y=S/S^{\star}-1\geq-1$. For $0<\varepsilon\leq1/2$, convexity makes $(1-\varepsilon)S^{\star}+\varepsilon S$ admissible, so
\(
 \E_Q\!\left[\frac{\log(1+\varepsilon Y)}{\varepsilon}\right]\leq0.
\)
The integrand tends pointwise to $Y$ as $\varepsilon\downarrow0$ and is bounded below by $\varepsilon^{-1}\log(1-\varepsilon)\geq-2$. Fatou's lemma therefore gives $\E_Q[Y]\leq0$, which is (i). 

Finally suppose (iii). Apply \eqref{eq:identity} to $(1-\varepsilon)S^{\star}+\varepsilon S$, so that  $R = 1+\varepsilon Y$, and hence
\[
 \E_Q[\min\{1+\varepsilon Y,2\}]\leq1,
 \qquad\text{or equivalently}\qquad
 \E_Q[\min\{Y,\varepsilon^{-1}\}]\leq0.
\]
Since $Y\geq-1$, monotone convergence after adding one gives $\E_Q[Y]\leq0$ as $\varepsilon\downarrow0$, which yields (i).
\end{proof}

\begin{remark}[Why relative log-optimality matters]\label{rem:relative}
If a finite expected-log optimum is attained, ordinary maximization of $\E_Q[\log S]$ gives the relative formulation under the usual well-definedness conditions. But relative log-optimality is meaningful even when separate expected logarithms are infinite. This point is already explicit in \citet[Theorem~2]{BellCover1988} and discussed again in~\cite{LarssonEtAl2025}; allowing infinite \emph{expected logarithms} should not be confused with allowing the wealth itself to be infinite.
\end{remark}

\begin{remark}[The above equivalence does not establish existence of a numeraire]
The convex class of constant wealths $\Ccal=\{s:0<s<1\}$ has no numeraire: for each candidate $s$, a larger admissible constant $t$ has $t/s>1$. Thus a general convex-family characterization is not a substitute for \cref{thm:lrr}.
\end{remark}

% \subsection{Recovering the portfolio theorem}

% Let $X=(X_1,\ldots,X_m)\geq0$ be the vector of gross returns under a known law $Q$, and let $\Delta_m=\{b\geq0:\sum_i b_i=1\}$. Suppose $b^{\star}$ attains a finite expected-log optimum over $b\in\Delta_m$, and put $S^{\star}=b^{\star\top}X$. The portfolio wealths form a convex class, so
% \[
%  \E_Q\!\left[\frac{b^{\top}X}{b^{\star\top}X}\right]\leq1,
%  \qquad b\in\Delta_m.
% \]
% More generally, let a random investment vector $B\geq0$ be independent of the market, with $\E[\sum_iB_i]\leq1$. Applying the preceding inequality to the individual stocks yields
% \[
%  \E_Q\!\left[\frac{B^{\top}X}{b^{\star\top}X}\right]
%  =\sum_i\E[B_i]\E_Q\!\left[\frac{X_i}{b^{\star\top}X}\right]\leq1.
% \]
% Therefore an independent  $U$ gives
% \begin{equation}\label{eq:portfolio}
%  \Pr_Q\bigl(Ub^{\star\top}X>B^{\top}X\bigr)\geq\frac12.
% \end{equation}
% This is the Bell--Cover conclusion. Notice that it is a pairwise comparison against each admissible opponent, not a simultaneous comparison against their hindsight maximum.

\subsection{The e-variable competitive theorem}\label{sec:egame}

\begin{theorem}[Bell--Cover randomization of a numeraire]\label{thm:egame}
Suppose $\num$ satisfies \eqref{eq:num} and \eqref{eq:finite}, and draw $U\sim\Unif(0,2)$ independently of the data and the opponent's seed. Then $U\num$ is a randomized e-variable for $\Pcal$, and for every privately randomized e-variable $T$ satisfying \eqref{eq:random-e},
\begin{equation}\label{eq:ehalf}
 \Pr_Q(U\num>T)\geq\frac12.
\end{equation}
The associated two-player game, with half credit for ties, has value $1/2$. A saddle point is obtained when both players use $\num$ multiplied by independent copies of $U$.
\end{theorem}
\begin{proof}
For every $P\in\Pcal$, 
\(
 \E_{P,U}[U\num]=\E[U]\E_P[\num]\leq1.
\)
For a randomized competitor, \eqref{eq:average} and the numeraire property give
\begin{equation}\label{eq:randomratio}
 \E_{Q,\rho}\!\left[\frac{T}{\num}\right]
 =\E_Q\!\left[\frac{\overline T}{\num}\right]\leq1.
\end{equation}
The uniform comparison identity proves \eqref{eq:ehalf}.
To state the minimax claim, let
\[
 h(a,b)=\ind\{a>b\}+\tfrac12\ind\{a=b\},
 \qquad H_Q(A,B)=\E_Q[h(A,B)].
\]
For independent private randomizations, $H_Q(A,B)+H_Q(B,A)=1$. The result already proved shows that playing $U\num$ guarantees at least $1/2$, while playing against an opponent $U'\num$ guarantees at most $1/2$. With both strategies, their common positive finite factor cancels and $\Pr(U>U')=1/2$. Consequently
\[
 \sup_A\inf_B H_Q(A,B)=\inf_B\sup_A H_Q(A,B)=\frac12.
\]
\end{proof}

% This theorem combines the Bell--Cover mechanism with the existence input of \citet{LarssonEtAl2025}. 
% It requires no representation of e-variables as stock portfolios. 
For any convex betting class, the same proof applies whenever that class has a numeraire, whose existence needs its own justification (as we have seen previously).

% \subsection{A likelihood-ratio example}

% For a simple null $\Pcal=\{P\}$ and $Q\ll P$, let $L=dQ/dP$. Then $L$ is a numeraire. Indeed, with $\{L>0\}$ understood up to $P$-null sets,
% \begin{equation}\label{eq:likelihood}
%  \E_Q[E/L]=\E_P[E\ind\{L>0\}]\leq\E_P[E]\leq1.
% \end{equation}
% The indicator is useful when $P$ and $Q$ are not equivalent. Thus $UL$ beats every $P$-valid privately randomized e-variable at least half the time under $Q$.

\subsection{What happens when the numeraire is infinite?}\label{sec:infinity}

The unrestricted existence theorem uses $x/\infty=0$ for finite $x$ and $\infty/\infty=1$. A strict-win statement cannot ignore this issue. For example, take $P=\delta_0$, $Q=\delta_1$, and an e-variable with $\num(0)=1$, $\num(1)=\infty$. Against the competitor $E=\num$, $U\num=\num=\infty$ under $Q$, so the strict-win probability is zero, not one half.

The tie-adjusted formulation does extend without the finiteness assumption, as shown below.

\begin{proposition}[Extended-value game]\label{prop:infinite}
With the extended-value numeraire of \citet{LarssonEtAl2025}, $U\num$ still guarantees
\[
 \E_Q[h(U\num,T)]\geq\frac12
\]
against every independently privately randomized e-variable $T$. Independent randomized numeraires form a saddle point with value $1/2$.
\end{proposition}
\begin{proof}
Define $R=T/\num$ using the extended ratio convention. On $\{\num<\infty\}$, averaging over the opponent's seed gives $\int R\,d\rho=\overline T/\num$. On $\{\num=\infty\}$,
\[
 \int R\,d\rho=\rho\{z:T(\omega,z)=\infty\}
 \leq\ind\{\overline T=\infty\}=\overline T/\num.
\]
Thus $\E_Q[R]\leq1$. For finite positive $\num$, conditioning on the data and the opponent's seed gives conditional payoff $1-\min\{R/2,1\}$. The same expression holds when $\num=\infty$: it is one if $T$ is finite and one half if $T=\infty$. Integrating proves the lower bound. Against another independent randomized numeraire, the conditional expected payoff is one half on both the finite and infinite regions. Symmetry gives the saddle-point statement.
\end{proof}

% The stronger strict-win wording used in the rest of the paper therefore has a clear scope, while the modern existence theorem retains a competitive interpretation in its full generality.

\section{The subtleties of optional stopping: validity vs.\ competitive optimality}\label{sec:validity}

We work in discrete time with a filtration $(\Fcal_t)_{t\geq0}$ and a trivial initial data sigma-field. An \emph{e-process} is a nonnegative adapted process $(E_t)$, with $E_0=1$, such that
\begin{equation}\label{eq:eprocess}
 \E_P[E_\tau]\leq1
 \quad\text{for every bounded stopping time $\tau$ and every $P\in\Pcal$}.
\end{equation}
Fatou's lemma extends this bound to almost surely finite stopping times. For a stopping time that may be infinite, the corresponding bound on $E_\tau\ind\{\tau<\infty\}$ also follows by Fatou. 
% The above  requirement is stronger than requiring an e-variable at each deterministic time; see \citet[Chapter~7]{RamdasWang2025}.

\begin{theorem}[Independent initial-capital randomization]\label{thm:processvalid}
Let $(E_t)$ satisfy \eqref{eq:eprocess}. Let $A\geq0$ have a fixed law $\nu$, independent of the entire data process under every $P\in\Pcal$, and satisfy $\E[A]\leq1$. On the product experiment, put $\widetilde E_t=AE_t$ and reveal $A$ at time zero. Then
\begin{equation}\label{eq:randomstop}
 \E_{P,\nu}[\widetilde E_\tau]\leq1
\end{equation}
for every bounded stopping time in the enlarged filtration $\mathcal G_t=\Fcal_t\vee\sigma(A)$. In particular, for $0<\alpha<1$,
\begin{equation}\label{eq:randomville}
 \Pr_{P,\nu}\!\left(\exists t\geq0:AE_t\geq1/\alpha\right)\leq\alpha.
\end{equation}
\end{theorem}
\begin{proof}
We work in the product space of  the data and lottery $A$. For a fixed $a$, the section $\tau_a$ of a bounded $\mathcal G$-stopping time is an $\Fcal$-stopping time. Tonelli and \eqref{eq:eprocess} give
\[
 \E_{P,\nu}[AE_\tau]
 =\int a\,\E_P[E_{\tau_a}]\,\nu(da)
 \leq\int a\,\nu(da)\leq1.
\]
Let $\gamma=\inf\{t:AE_t\geq1/\alpha\}$. Apply \eqref{eq:randomstop} at $\gamma\wedge n$ and use nonnegativity to obtain $\Pr(\gamma\leq n)\leq\alpha$. Let $n$ increase to infinity.
\end{proof}

This theorem does not require a numeraire, log-optimality, or even a supermartingale representation of the original e-process. It permits the stopping rule to inspect the realized lottery $A$. If one insists that the displayed e-process start deterministically at one, prepend a preliminary lottery bet:
\[
 1\ \longrightarrow\ A\ \longrightarrow\ AE_1\ \longrightarrow\ AE_2\ \longrightarrow\cdots.
\]
The same argument gives a conventional initial-one e-process on that expanded time scale.

For $A=U$, validity is \emph{ex ante}, not conditional at level $\alpha$ for every lottery outcome. Conditional on $U=u$, the ordinary crossing argument gives the bound $\min\{\alpha u,1\}$. Values $u>1$ spend more than the nominal conditional error budget, compensated by values $u<1$. Selecting a favorable draw, or redrawing until the capital exceeds one, changes the experiment and invalidates the above justification.

\subsection{With stopping, competitive optimality and validity are different issues}\label{sec:sequential}

The previous result showed that we get a \emph{valid} e-process even when the stopping time can access $U$. However, the competitive optimality result below only holds if the stopping time is independent of $U$.

% \subsection{A stopped-ratio criterion}

\begin{proposition}[Competition at an independent evaluation]\label{prop:stopped}
Let $\tau$ be a comparison stopping time and let $E_\tau^\star$ be positive and finite $Q$-almost surely. Suppose that, for every competitor under consideration,
\begin{equation}\label{eq:stoppedratio}
 \E_Q[E_\tau/E_\tau^\star]\leq1.
\end{equation}
If $U\sim\Unif(0,2)$ is independent of the evaluated pair $(E_\tau,E_\tau^\star)$, then
\begin{equation}\label{eq:stoppedhalf}
 \Pr_Q(UE_\tau^\star>E_\tau)\geq\frac12.
\end{equation}
\end{proposition}
\begin{proof}
Apply \cref{lem:uniform} to the stopped wealths.
\end{proof}

The criterion allows data-dependent stopping. It also allows a stopping rule to use an opponent's independent private seed, provided \eqref{eq:stoppedratio} holds on that experiment. What is excluded is dependence on the specific $U$ used in the comparison. Drawing a fresh $U$ after the stopping decision ensures the required independence.

% A standard sufficient condition for \eqref{eq:stoppedratio} is that $E_t/E_t^\star$ is a nonnegative $Q$-supermartingale starting at one. Optional stopping then gives the criterion at bounded stopping times, and Fatou gives it at $Q$-almost surely finite stopping times. There need not, however, be a single benchmark with this property against all e-processes for an arbitrary composite null.

\subsection{A simple extension: conditional log-optimal betting}

Here is the familiar setting in which the ratio-supermartingale argument works. At each time $t$, let the allowed wealth multipliers form a conditionally convex class of nonnegative $\Fcal_t$-measurable variables $f_t$, with
\[
 \E_P[f_t\mid\Fcal_{t-1}]\leq1\qquad(P\in\Pcal).
\]
Assume that a positive, finite, measurable choice $f_t^\star$ exists within this class and satisfies
\[
 \E_Q[\log(f_t/f_t^\star)\mid\Fcal_{t-1}]\leq0
 \quad\text{for every allowed $f_t$},
\]
with well-defined conditional expectations. This is conditional relative log-optimality. The conditional version of the convex-perturbation proof in \cref{thm:equivalence} gives
\begin{equation}\label{eq:conditionalratio}
 \E_Q\!\left[\frac{f_t}{f_t^\star}\,\middle|\,\Fcal_{t-1}\right]\leq1.
\end{equation}
Consequently $E_t=\prod_{s=1}^t f_s$ is a null test supermartingale, and
\[
 R_t=\frac{E_t}{E_t^\star}
 =\prod_{s=1}^t\frac{f_s}{f_s^\star}
\]
is a nonnegative $Q$-supermartingale. This proves competitive optimality at stopping times independent of $U$ within the specified betting class.

The multiperiod mechanism is already present in \citet[Theorem~4]{BellCover1988}. Their Section~6 also discusses conditional games and the sufficiency of terminal randomization. We are applying that mechanism to conditional e-bets. 
% The assumption about  existence of the conditional numeraire is substantive: the class of products of available conditional bets need not equal the class of all e-processes for a composite null.

\subsection{A simple example: likelihood-ratio processes against a simple null}

For a simple null $P$, let $Q|_{\Fcal_t}\ll P|_{\Fcal_t}$ for every finite $t$, and define
\[
 L_t=\frac{dQ|_{\Fcal_t}}{dP|_{\Fcal_t}}.
\]
The likelihood-ratio process is a $P$-martingale with initial value one and is positive and finite $Q$-almost surely at finite times. If $(E_t)$ is \emph{any} $P$-e-process, not necessarily a product of one-step bets, and $\tau$ is a bounded data stopping time, change of measure on each event $\{\tau=t\}$ gives
\begin{align}
 \E_Q[E_\tau/L_\tau]
 &=\sum_t\E_P\!\left[E_t\ind\{\tau=t,L_t>0\}\right]\notag\\
 &\leq\E_P[E_\tau]\leq1.\label{eq:lrstop}
\end{align}
For a $Q$-almost surely finite $\tau$, the same sum is bounded by $\E_P[E_\tau\ind\{\tau<\infty\}]\leq1$, with the latter inequality obtained from bounded stopping and Fatou. Thus the conclusion does not require global absolute continuity of the laws on the infinite-horizon sigma-field. Independent private randomization can be included by taking product measures with the common seed law.

It follows that $UL_\tau$ wins against every $P$-e-process at least half the time under $Q$, at every such comparison time independent of $U$. The related stopped log-optimality of likelihood ratios is stated in \citet[Theorem~7.11]{RamdasWang2025}. Equation~\eqref{eq:lrstop} does \emph{not} assert that $E_t/L_t$ is a $Q$-supermartingale for every e-process; only the stopped expectation is needed.

\subsection{Three ways a stronger claim can fail}

\begin{example}[The comparison time sees $U$]\label{ex:ustop}
Let $P=Q$ be a fair-coin law. Take the benchmark $E_t^\star\equiv1$ and the competitor $R_0=1$, with $R_1=2$ on heads and $R_1=0$ on tails; keep both processes constant after time one. They are valid test martingales, and the benchmark is the numeraire at every deterministic time. Draw an independent $U\sim\Unif(0,2)$ and stop at
\[
 \tau=\begin{cases}0,&U<1,\\1,&U\geq1.\end{cases}
\]
On $\{U<1\}$ the randomized benchmark immediately loses. On $\{U\geq1\}$ it wins only on tails. Therefore
\[
 \Pr_Q(UE_\tau^\star>R_\tau)=\frac14.
\]
Nevertheless, \cref{thm:processvalid} still applies to null error control. This example separates competitive optimality from e-process validity and from a simultaneous guarantee of staying ahead at all times.
\end{example}

\begin{example}[The opponent observes $U$]\label{ex:public}
Even with no data and benchmark $E^\star=1$, if an opponent's randomization is allowed to depend on the same $U$, they can use the randomized initial wealth equal to
\[
 T=(U+1/4)\ind\{U\leq3/2\}.
\]
It has $\E[T]=3/4\leq1$ but beats $U$ with probability $3/4$. Thus admissible ex ante mean alone does not protect a player who reveals its private seed before the opponent chooses its response. The independent-private-randomization protocol is part of the theorem.
\end{example}

\begin{example}[Marginal numeraires do not automatically form an e-process]\label{ex:marginal}
Let $\Omega=\{a,b,c\}$, let $\Fcal_1$ distinguish $\{a\}$ from $\{b,c\}$, and let $\Fcal_2$ reveal the outcome completely. In the coordinate order $(a,b,c)$, take
\[
 P_1=(3/4,1/4,0),\qquad P_2=(1/4,0,3/4),\qquad Q=(1/2,1/2,0).
\]
For $\Pcal=\{P_1,P_2\}$, the time-one numeraire is $E_1^\star=1$. Indeed, an $\Fcal_1$-measurable e-variable has the form $(x,y,y)$ and
\[
 \E_Q[E]=\tfrac12(x+y)=\tfrac12\E_{P_1}[E]+\tfrac12\E_{P_2}[E]\leq1.
\]
At time two, a numeraire is $E_2^\star=(2/3,2,0)$. Its null means are $1$ and $1/6$, and every e-variable satisfies
\[
 \E_Q[E/E_2^\star]=\tfrac34E(a)+\tfrac14E(b)=\E_{P_1}[E]\leq1.
\]
But stopping at time one on $\{a\}$ and time two otherwise yields
\[
 \E_{P_1}[E_\tau^\star]=\tfrac34\cdot1+\tfrac14\cdot2=\frac54>1.
\]
Thus $E_0^\star=1,E_1^\star,E_2^\star$ does not form an e-process, despite exact static optimality at both observation times. The general distinction is discussed in \citet[Section~7.9]{RamdasWang2025}.
\end{example}

\section{Important  caveats}\label{sec:costs}

\subsection{Risks of the competitive objective}

The competitive objective can strongly prefer a strategy that usually wins a little and occasionally loses everything. For $0<\varepsilon<1$, let
\[
 V=\begin{cases}
 (1-\varepsilon)^{-1},&\text{with probability }1-\varepsilon,\\
 0,&\text{with probability }\varepsilon.
 \end{cases}
\]
This is a fair lottery: $\E[V]=1$. Accordingly, $V\num$ is a valid randomized e-variable, and it beats the unrandomized $\num$ with probability $1-\varepsilon$. The probability can be arbitrarily close to one even though $\num$ is log-optimal, but it risks bankruptcy so has expected log-wealth equal to minus infinity. This is the fair-lottery phenomenon discussed by \citet[Section~2]{BellCover1980}. The uniform strategy protects against precisely this vulnerability in a probability-of-outperformance game.

\subsection{Randomization is only guaranteed to help the numeraire}

Randomization does not improve the competitive win probability of an arbitrary e-variable! 
For example, with a fair-coin null and an alternative concentrated on heads, $E=2\ind\{\text{heads}\}$ is an e-variable and equals two almost surely under the alternative. Randomizing the constant-one e-variable gives $U<2$ almost surely, so it loses to $E$ with probability one under $Q$. In addition, the optimal numeraire generally depends on $Q$: competitive optimality for a specified alternative is not a universal finite-sample guarantee against every alternative.

We attempted to generalize the above phenomena to asymptotic log-optimality (using the mixture or plug-in methods, and employing regret bounds) when $Q$ is not known, but none of the results we obtained were fundamentally different from the presented ones: they just separately combined the competitive optimality of the numeraire with an appropriate argument showing that one can often bound the wealth of a strategy from that of the numeraire. Thus, we did not include such results in the paper for simplicity, since they added much algebra but little new insight.

\subsection{The randomized numeraire is not the numeraire}

Enlarging the experiment by an independent ancillary seed does not dislodge $\num$ from its numeraire role. The averaging argument in \cref{sec:egame} shows that, for every randomized e-variable $T$,
\[
 \E_Q[T/\num]\leq\E_Q[\overline T/\num]\leq1,
\]
with equality in the first step under \eqref{eq:finite}. By contrast, in that finite-positive regime, the randomized wealth $U\num$ itself is not a numeraire on the enlarged experiment. Indeed, the original $\num$ remains admissible, but
\begin{equation}\label{eq:inverseuniform}
 \E_Q\!\left[\frac{\num}{U\num}\right]=\E[U^{-1}]=\infty.
\end{equation}
Thus the numeraire optimum and the randomized competitive optimum are different objects. 
% Neither criterion should be silently substituted for the other.

\subsection{An exact log-evidence penalty}

The relative logarithmic cost is particularly simple:
\begin{equation}\label{eq:logcost}
 \E_Q\!\left[\log\frac{U\num}{\num}\right]
 =\E[\log U]
 =\frac12\int_0^2\log u\,du
 =\log2-1\approx-0.30685.
\end{equation}
This relative calculation remains meaningful even if the separate expected logarithms of the two e-variables are infinite. When separate expectations are finite, it says that the randomization reduces expected log evidence by approximately $0.307$ nats, corresponding to a geometric multiplier $2/e\approx0.736$.

Without considering expectations, the randomization can at most double the underlying wealth, while its negative log effect is unbounded because $U$ can be arbitrarily close to zero. Against the underlying wealth itself, it increases the reported evidence half the time and decreases it half the time. It therefore provides neither stochastic dominance nor a pointwise improvement in evidence.

A \emph{single} initial draw does not change an asymptotic logarithmic growth rate: along positive-wealth paths,
\[
 \frac{\log(UE_t)-\log E_t}{t}=\frac{\log U}{t}\longrightarrow0.
\]
Repeating the randomization at every round is a different procedure. In a product-betting model where independent mean-one draws preserve conditional null validity, independent $U_1,U_2,\ldots$ impose
\[
 \frac1t\log\prod_{s=1}^t U_s\longrightarrow\log2-1
 \quad\text{almost surely}.
\]
Thus, for repeated randomization, the strong law turns a harmless asymptotic one-time offset into a persistent negative growth rate. Thus one should not insert such an extra randomization at every bet.

\section{Terminal testing: power and sharp calibration}\label{sec:terminal}

Suppose $E$ is any e-variable, not necessarily a numeraire, and $U\sim\Unif(0,2)$ is independent. Fix $0<\alpha<1$ and write $z=\alpha E$. Conditional on $E$, the ordinary test and the test based on the randomized capital have rejection probabilities
\begin{align}
 r_{\ord}(z)&=\ind\{z\geq1\},\label{eq:rordinary}\\
 r_{\comp}(z)&=\Pr_U(UE\geq1/\alpha\mid E)
 =\begin{cases}\pos{1-\dfrac1{2z}},&z>0,\\0,&z=0.\end{cases}\label{eq:rcover}
\end{align}
At $z=\infty$, all randomized rejection functions considered below are defined to equal one, their limiting value. The lottery $U$ increases the chance of rejection for $1/2<z<1$, but removes some rejections whenever $z\geq1$. At the original threshold $z=1$, rejection probability falls from one to one half, as we formalize below. Therefore it does not uniformly improve power, even before considering alternative-specific choices of $E$.

\subsection{Factor of two level conservativeness at fixed times}

% There is a sharper calibration fact than the generic e-variable bound.

\begin{proposition}[Sharp terminal conservativeness]\label{prop:halfsize}
For every e-variable $E$ independent of $U$,
\begin{equation}\label{eq:halfsize}
 \sup_{P\in\Pcal}\Pr_{P,U}(UE\geq1/\alpha)\leq\alpha/2.
\end{equation}
The constant is sharp over null distributions and e-variables. The same conclusion holds with $E=E_\tau$ for a null-valid stopped e-process evaluated at a stopping time independent of $U$.
\end{proposition}
\begin{proof}
For all $z\geq0$, $r_{\comp}(z)\leq z/2$. It is immediate when $z\leq1/2$; otherwise
\[
 \frac z2-r_{\comp}(z)=\frac{(z-1)^2}{2z}\geq0.
\]
Consequently $\E_P[r_{\comp}(\alpha E)]\leq(\alpha/2)\E_P[E]\leq\alpha/2$. For sharpness, take an event of null probability $\alpha$ and put $E=1/\alpha$ on that event and zero otherwise. This e-variable has mean one, and the randomized test rejects with probability $\alpha\Pr(U\geq1)=\alpha/2$.
\end{proof}

Thus the usual threshold $1/\alpha$ is conservative by a factor of two for this particular terminal product, even though the product itself has mean at most one. The uniform distribution spreads the available mass in a way that makes Markov's inequality non-sharp at one independent evaluation.

\begin{corollary}[A terminal recalibration]\label{cor:calibrated}
For $0<\alpha\leq1/2$, the terminal rule that rejects when
\begin{equation}\label{eq:calibrated}
 UE\geq\frac1{2\alpha}
\end{equation}
is level $\alpha$, sharply over null distributions and e-variables. Its conditional rejection probability is
\begin{equation}\label{eq:rcalibrated}
 r_{\calcomp}(z)=\begin{cases}\pos{1-\dfrac1{4z}},&z>0,\\0,&z=0.\end{cases}
\end{equation}
\end{corollary}
\begin{proof}
The same calculation, with $2\alpha$ in place of $\alpha$, gives the bound. Equality holds when $E=1/(2\alpha)$ with null probability $2\alpha$ and zero otherwise.
\end{proof}

This is a \emph{testing calibration}, not an assertion that $2UE$ is an e-variable. In general its null mean can be two. It is also a terminal statement, whose naïve extension to continuous monitoring would be incorrect.

\subsection{Factor of two conservativeness disappears under optional stopping}

\begin{proposition}[The anytime bound remains sharp]\label{prop:anytimesharp}
For every $0<\alpha\leq1/2$, the bound $\alpha$ in \eqref{eq:randomville}, with $A=U$, can be approached arbitrarily closely by some nonnegative null martingale starting at one. 
% There is no general anytime improvement to $\alpha/2$.
\end{proposition}
\begin{proof}
Fix $r>1$. Under the null let $Z_1,Z_2,\ldots$ be independent Bernoulli variables with success probability $1/r$, and put
\[
 M_0=1,\qquad M_n=r^n\ind\{Z_1=\cdots=Z_n=1\}.
\]
This is a nonnegative mean-one martingale. Draw $U$ independently. For $u\in(0,2)$, the smallest number of successes needed to cross $1/\alpha$ is
\(
 k(u)=\left\lceil\log_r\frac1{\alpha u}\right\rceil.
\)
Since $\alpha u\leq1$, the conditional crossing probability is $r^{-k(u)}$, and
\[
 \frac{\alpha u}{r}\leq r^{-k(u)}\leq\alpha u.
\]
Averaging over $U$ gives
\[
 \frac\alpha r\leq\Pr\!\left(\exists n:UM_n\geq1/\alpha\right)\leq\alpha.
\]
Letting $r\downarrow1$ completes the proof.
\end{proof}

The first-crossing time depends on $U$, so \cref{prop:halfsize} cannot be applied there. For the lowered threshold $1/(2\alpha)$ and $\alpha\leq1/4$, the same construction with $2\alpha$ in place of $\alpha$ gives an anytime rejection probability arbitrarily close to $2\alpha$. Therefore the terminal recalibration in \eqref{eq:calibrated} must not be monitored indefinitely at level $\alpha$.

\section{Randomizing a test threshold instead of initial capital}\label{sec:threshold}

\subsection{Comparing to the uniformly randomized Markov's inequality}

A distinct use of a uniform variable is statistically more favorable when the objective is power from a given terminal e-variable. Let $V\sim\Unif(0,1)$ be independent and reject when
\begin{equation}\label{eq:umi}
 E\geq V/\alpha.
\end{equation}
Its conditional rejection probability is
\begin{equation}\label{eq:rumi}
 r_{\umi}(z)=\min\{z,1\},\qquad z=\alpha E.
\end{equation}
Thus its null rejection probability is at most $\alpha\E_P[E]\leq\alpha$. This is an application of the uniformly randomized Markov inequality developed in \citet{RamdasManole2026}. It retains every rejection of the ordinary test and can add rejections below the ordinary threshold.

\begin{proposition}[Power comparison with both Bell-Cover calibrations]\label{prop:dominance}
For every $z\geq0$,
\begin{equation}\label{eq:dominance}
 r_{\comp}(z)\leq r_{\calcomp}(z)\leq r_{\umi}(z).
\end{equation}
Consequently, the uniformly randomized Markov test has at least as much power, under every data distribution, as either the ordinary-threshold Bell-Cover test or its sharp terminal recalibration. This compares tests based on the \emph{same underlying} $E$.
\end{proposition}
\begin{proof}
The first inequality follows from the thresholds. For the second, the cases $z\leq1/4$ and $z\geq1$ are immediate. For $1/4<z<1$,
\[
 r_{\umi}(z)-r_{\calcomp}(z)
 =z-1+\frac1{4z}=\frac{(2z-1)^2}{4z}\geq0.
\]
Taking expectations proves the power comparison. There is also a pointwise coupling: set $V=1-U/2$, so $V$ is uniform on $(0,1)$. For every $u\in(0,2)$,
\[
 \frac1{2u}-\left(1-\frac u2\right)=\frac{(u-1)^2}{2u}\geq0.
\]
Hence $E\geq1/(2\alpha U)$ implies $E\geq V/\alpha$ on every data-and-lottery realization. The ordinary-threshold Bell-Cover test has an even smaller rejection event under this coupling.
\end{proof}

The dominance survives the removal of the Bell-Cover test's terminal size conservativeness. It is therefore not just a comparison between a level-$\alpha/2$ test and a level-$\alpha$ test. Some representative conditional rejection probabilities are shown in \cref{tab:power}.

\begin{table}[htbp]
\centering
\caption{Conditional rejection probability given $z=\alpha E$. ``Calibrated'' refers to using the terminal threshold $1/(2\alpha)$ for $UE$.}
\label{tab:power}
\begin{tabular}{@{}ccccc@{}}
\toprule
$z$ & Ordinary & Bell-Cover & Bell-Cover, calibrated & Uniform threshold\\
\midrule
$1/4$ & $0$ & $0$ & $0$ & $1/4$\\
$1/2$ & $0$ & $0$ & $1/2$ & $1/2$\\
$3/4$ & $0$ & $1/3$ & $2/3$ & $3/4$\\
$1$   & $1$ & $1/2$ & $3/4$ & $1$\\
$2$   & $1$ & $3/4$ & $7/8$ & $1$\\
\bottomrule
\end{tabular}
\end{table}

There is no contradiction with competitive optimality. The threshold procedure is a randomized decision rule, not a fair initial-capital lottery. Writing it as $E/V\geq1/\alpha$ does not make $E/V$ an e-variable: if $\E_P[E]>0$, Tonelli and $\E[V^{-1}]=\infty$ give $\E_{P,V}[E/V]=\infty$. It is the rejection event, not a putative wealth process $E/V$, that is calibrated.

\subsection{A valid sequential use of terminal threshold randomization}

Let $(E_t)$ be an e-process. Monitor its original threshold $1/\alpha$, and let $\tau$ be a bounded data stopping time representing the end of the study. At that time, draw a fresh independent $V$ for a final randomized decision. The combined rule
\begin{equation}\label{eq:randomizedville}
 \left\{\exists t<\tau:E_t\geq1/\alpha\right\}
 \ \cup\ \{E_\tau\geq V/\alpha\}
\end{equation}
is level $\alpha$. Indeed, let $\gamma=\inf\{t:E_t\geq1/\alpha\}$ and $\sigma=\gamma\wedge\tau$. The event in \eqref{eq:randomizedville} equals $\{E_\sigma\geq V/\alpha\}$: if $\gamma<\tau$, then $E_\sigma\geq1/\alpha$ makes the latter rejection certain. Therefore its probability is
\[
 \E_P[\min\{\alpha E_\sigma,1\}]\leq\alpha\E_P[E_\sigma]\leq\alpha.
\]
The argument extends to almost surely finite study endpoints. This is the e-process form of the randomized Ville inequality in \citet[Section~4]{RamdasManole2026}.

By contrast, even a \emph{single} independent draw $V$ does not justify monitoring the lower threshold $V/\alpha$ forever. For a concrete example, let $E_0=1$ and let $E_1$ equal two or zero with equal null probabilities. For $\alpha\leq1/2$, monitoring the lowered threshold at both times rejects if $V\leq\alpha$, or if the coin is heads and $\alpha<V\leq2\alpha$. Its size is $\alpha+\alpha/2=3\alpha/2$. Fresh uniforms at repeated looks can be still worse. Terminal threshold randomization and independent initial-capital randomization therefore have fundamentally different optional-stopping behavior.

\section{Conclusion}\label{sec:conclusion}

Bell and Cover's randomization of the initial capital has an exact and useful interpretation in testing by betting. Against a specified $Q$, an independently randomized numeraire wins a pairwise comparison with every independently randomized e-variable at least half the time. For finite positive numeraires this is a strict-win statement; unrestricted extended-valued numeraires retain the corresponding half-tie minimax statement. The uniform lottery is an optimal response to a probability-of-outperformance objective, not a modification required for e-validity.

For sequential testing, an independent mean-one initial lottery preserves null validity under stopping rules that observe the lottery. Competitive optimality is more delicate: it requires a stopped-ratio bound and a comparison protocol that does not exploit that lottery. Conditional log-optimal betting and simple-null likelihood ratios provide important positive cases.
% ; arbitrary composite-null marginal numeraires do not automatically provide an e-process.

The costs are equally explicit. For finite positive numeraires, the initial lottery sacrifices $1-\log2$ nats of expected relative log evidence and abandons the numeraire property of the reported wealth. It does not confer (improvement or) competitive optimality on an arbitrary betting strategy. At the usual terminal threshold it has sharp size at most $\alpha/2$, while anytime monitoring can (in some cases) use the full $\alpha$. Even after a sharp terminal recalibration, it is dominated in conditional rejection probability by the uniformly randomized Markov test based on the same underlying e-variable. This latter procedure randomizes a decision threshold rather than fair betting capital and must be used with its own sequential safeguards.

Thus three objectives should remain separate. For a Bell--Cover head-to-head game, the uniform-capital lottery is exactly appropriate. For logarithmic evidence accumulation, the unrandomized numeraire is the benchmark. For power from a terminal e-variable, randomizing the testing threshold is the more favorable than the random initial wealth considered here. The distinctions are compatible with Bell and Cover's own caution about recommending their lottery in practice. 
% and with the genuine additional scope of the modern numeraire e-variable theory.

% The Appendix explores some implications for the setting where the alternative $\mathcal Q$ is composite, and the method of mixtures or the plug-in method are employed to achieve asymptotic log-optimality; here the story is much less simple, and the details are thus omitted from the main paper.

\paragraph{Acknowledgments.} The author acknowledges the use of GPT 6 in turning a set of initial notes into  an initial paper draft, whose writing and math was then significantly modified and made crisper. Several of the main results for e-variables were proved by the author, but GPT suggested interesting examples, and was also helpful in brainstorming additional results. The author has checked all details for correctness and relevance, and takes responsibility for all content.

\bibliographystyle{plainnat}
\bibliography{references}
\end{document}